\documentclass{amsart}
\usepackage[utf8]{inputenc}
\usepackage[english]{babel}
\usepackage[T1]{fontenc}

\usepackage{amsmath}
\usepackage{amsfonts}
\usepackage{amssymb}
\usepackage{amsthm} 
\usepackage{amsaddr}
\usepackage{hyperref} 
\usepackage{stmaryrd} 
\usepackage{yhmath} 
\usepackage{tikz}
\usepackage{subcaption}
\usepackage{accents} 
\usepackage{xcolor} 
\usepackage{dsfont} 

\usepackage{blindtext}
\usepackage{enumitem} 

\usepackage[left=2cm,right=2cm,top=2cm,bottom=2cm]{geometry}
\usepackage{verbatim}

\DeclareMathOperator{\R}{\mathbb{R}}

\DeclareMathOperator{\E}{\mathbb{E}}
\DeclareMathOperator{\1}{\mathds{1}}

\DeclareMathOperator{\supp}{\mathrm{supp}}

\DeclareFontFamily{U}{mathx}{}
\DeclareFontShape{U}{mathx}{m}{n}{<-> mathx10}{}
\DeclareSymbolFont{mathx}{U}{mathx}{m}{n}
\DeclareMathAccent{\widehat}{0}{mathx}{"70}
\DeclareMathAccent{\widecheck}{0}{mathx}{"71}

\usepackage[backend=biber,style=numeric]{biblatex}
\usepackage{csquotes}
\numberwithin{equation}{section}

\makeatletter
\def\blfootnote{\xdef\@thefnmark{}\@footnotetext}
\makeatother

\theoremstyle{plain}
\newtheorem{theorem}{Theorem}[section]
\newtheorem{remark}[theorem]{Remark}
\newtheorem{lemma}[theorem]{Lemma}
\newtheorem{proposition}[theorem]{Proposition}
\newtheorem{corollary}[theorem]{Corollary}
\newtheorem{hypothesis}[theorem]{Hypothesis}

\title[Modified scattering for Vlasov-Maxwell with small data]{Modified scattering for small data solutions to the Vlasov-Maxwell system: a short proof}
\author{Emile Breton}
\address{Univ Rennes, CNRS, IRMAR - UMR 6625, F-35000 Rennes, France \\ Email address : emile.breton@univ-rennes.fr}

\begin{document}


\begin{abstract}
    We prove that for any global solution to the Vlasov-Maxwell system arising from compactly supported data, and such that the electromagnetic field decays fast enough, the distribution function exhibits a modified scattering dynamic. In particular, our result applies to every small data solution constructed by Glassey-Strauss in \cite{Glassey_Strauss_1987}.
\end{abstract}



\maketitle

\blfootnote{This work was conducted within  the France 2030 program, Centre Henri Lebesgue ANR-11-LABX-0020-01.}


\section{Introduction and main results}

\subsection{General context}

The relativistic Vlasov-Maxwell system models a collisionless plasma, it can be written as 
\begin{equation}
    \tag{RVM}
    \label{equation_RVM1}
    \sqrt{m_\alpha^2+|v|^2} \partial_t f_\alpha+v\cdot\nabla_x f_\alpha +e_\alpha\Big(\sqrt{m_\alpha^2+|v|^2} E+v\times B\Big)\cdot\nabla_v f_\alpha=0,\qquad 1\leq \alpha \leq N,
\end{equation}
\begin{align*}
    \partial_tE=\nabla\times B -4\pi j,&\quad \nabla\cdot E=4\pi\rho,\\
    \partial_t B=-\nabla\times E,&\quad  \nabla\cdot B=0,
\end{align*}
where $\rho,j$ are the total charge and current density of the plasma defined by 
\begin{equation*}
    \rho=\sum_{1\leq\alpha\leq N} e_\alpha\int_{\R^3_v} f_\alpha \mathrm{d}v,\quad j=\sum_{1\leq\alpha\leq N}e_\alpha\int_{\R^3_v}\widehat{v_\alpha}f_\alpha\mathrm{d}v.
\end{equation*}
Here we consider the multi-species case $N\geq 2$, where $f_\alpha(t,x,v)$ is the density function of a species $\alpha$ with mass $m_\alpha>0$ and charge $e_\alpha\neq 0$. Here, $t\geq 0$ denotes the time, $x\in\R^3_v$ the particle position and $v\in$ the particle momentum. Moreover, $(E,B)(t,x)$ denotes the electromagnetic field of the plasma. For a particle of mass $1$ and momentum $v\in\R^3_v$, we will denote its energy by $v^0:=\langle v\rangle=\sqrt{1+|v|^2}$ and its relativistic speed as 
    \begin{equation}
        \widehat{v}:=\frac{v}{v^0},\quad v\in\R_v^3.
    \end{equation}
In addition, we denote $v_\alpha(v)=\frac{v}{m_\alpha}$. We will simply write $v_\alpha$ since there is no risk of confusion. Then  
    \begin{equation}
        \widehat{v_\alpha}=\frac{v}{\sqrt{m_\alpha^2+|v|^2}}=\frac{v}{v^0_\alpha},\qquad v^0_\alpha:=\sqrt{m_\alpha^2+|v|^2}.
    \end{equation}
Note that we have $v_\alpha^0=m_\alpha\langle v_\alpha\rangle$. Finally, the initial data $f_{\alpha0}=f_\alpha(0,\cdot)$ and $(E,B)(0,\cdot)=(E_0,B_0)$ also satisfy, in the electrically neutral setting, the constraint equations 
\begin{equation}
    \label{equation_condition_condition_initiale}
    \nabla\cdot E_0=4\pi\sum_{1\leq\alpha\leq N} e_\alpha\int_{\R^3_v}  f_{\alpha0} \mathrm{d}v,\qquad \nabla\cdot B_0=0,\qquad \sum_{1\leq\alpha\leq N} e_\alpha \int_{\R^3_x\times\R^3_v}f_{\alpha0} \mathrm{d}v\mathrm{d}x=0.
\end{equation}

In 3D the global existence problem for the classical solutions to \eqref{equation_RVM1} is still open, though various continuation criteria have been proved (see, for instance, \cite{Luk_Strain_14}).

The case of small data solutions was first studied by Glassey, Strauss, and Schaeffer \cite{Glassey_Strauss_1987,Glassey_Schaeffer_88}. They showed that the solutions to \eqref{equation_RVM1} arising from small and compactly supported data are global in time. The compact support assumption on the momentum variable $v$ was later removed by Schaeffer \cite{Schaeffer_04}.
More recently, without any compact support hypothesis, \cite{Wei_Yang_21,Bigorgne_20} established propagation of regularity for the small data solution to \eqref{equation_RVM1} and \cite{Wang_22} relaxed the smallness assumption on the electromagnetic field. Finally, a modified scattering dynamic was derived for the distribution function ; see \cite{Bigorgne_2023,Pankavich_BenArtzi_2024}, along with a scattering map \cite{bigorgne_2023_scatteringmap}.

Similar results have been obtained for the Vlasov-Poisson equation, for instance modified scattering has been proved for small data \cite{Choi_Kwon_16,Ionescu_Pausader_22,Pankavich2022,Flynn_Ouyang_Pausader_Widmayer_2023} (see also \cite{bigorgne_velozo_2024,Schlue_Taylor_2024} for more refinements). It was also shown that, in the single species case and for a non-trivial distribution function, linear scattering cannot occur \cite{Choi_Ha_11}. Note that modified scattering also holds in the context of stability of a point charge \cite{Pausader_Widmayer_2021,Pausader_Widmayer_Yang_2024}, or with the external potential $-\frac{|x|^2}{2}$ \cite{Bigorgne_Velozo_2025}.\\
In this paper, we provide a short proof of modified scattering for the distribution functions $f_\alpha$. Compared with Pankavich and Ben-Artzi \cite{Pankavich_BenArtzi_2024}, who also worked on solutions constructed by Glassey and Strauss in \cite{Glassey_Strauss_1987}, our approach does not require to assume more regularity on the data than in \cite{Glassey_Strauss_1987}.

\subsection{Main result}

We assume the following properties and derive the results in this context. 
\begin{hypothesis}
\label{hypothese_suffisante}
Assume $(f_\alpha,E,B)$ is a $C^1$ global solution to \eqref{equation_RVM1} with initial data $(f_{\alpha0},E_0,B_0)$ and satisfying the following properties.
    \begin{itemize}
        \item There exists $k>0$ such that $f_{\alpha0}$ are non-negative $C^1$ functions with support in $\{(x,v)\,|\,|x|\leq k,\,|v|\leq k\}$. Moreover, $E_0,B_0$ are $C^1$ with support in $\{x\,|\,|x|\leq k\}$ and satisfy the constraint \eqref{equation_condition_condition_initiale}.
        \item There exists $C_0>0$ such that, for all $(t,x)\in\R_+\times\R^3_x$,
        \begin{align}
            \label{equation_estimee_EB_Glassey}
            |(E,B)(t,x)|&\leq\frac{C_0}{(t+|x|+2k)(t-|x|+2k)},\\
            \label{equation_estimee_gradEB_Glassey}
            |\nabla_x(E,B)(t,x)|&\leq\frac{C_0\log(t+|x|+2k)}{(t+|x|+2k)(t-|x|+2k)^2}.
        \end{align}
    \end{itemize}
\end{hypothesis}
\begin{remark}
    In the following, we will write $a\lesssim b$ when there exists $C>0$, independent of $t$ and depending on $(x,v)$ only through $k$, such that $a\leq Cb$. However, here $C$ will usually depend on $(m_\alpha,e_\alpha)_{1\leq \alpha\leq N}$ and the initial data.
\end{remark}
\begin{remark}
    Since $(f_\alpha,E,B)$ is a solution to \eqref{equation_RVM1} with compactly supported data, it follows that 
    \begin{equation*}
        \supp (E,B)(t,\cdot)\subset \{x\in\R^3_x\,|\, |x|\leq t+k\}.
    \end{equation*}
    Hence the right hand sides of \eqref{equation_estimee_EB_Glassey} and \eqref{equation_estimee_gradEB_Glassey} are always bounded.
    
\end{remark}
\begin{remark}
    It is important to note that, according to \cite[Theorem 1]{Glassey_Strauss_1987} and \cite{Glassey_Schaeffer_88}, for small compactly supported data or nearly neutral data, the unique associated classical solution to \eqref{equation_RVM1} satisfies Hypothesis \ref{hypothese_suffisante}. This proves that, when the data are compactly supported, Theorem \ref{theoreme_principal_scattering_modifie} holds for small or nearly neutral data. Finally, our result applies to a subclass of the solutions constructed by Rein \cite{Rein_90} as well, those arising from compactly supported data.
\end{remark}
We now state that for any $(f_\alpha,E,B)$ verifying the above hypothesis, the distribution functions $f_\alpha$ satisfy a modified scattering dynamic. Moreover, we are able to prove that the asymptotic limits $f_{\alpha\infty}$ are compactly supported.

\begin{theorem}
    \label{theoreme_principal_scattering_modifie}
    Let $(f_{\alpha},E,B)$ be a solution of \eqref{equation_RVM1} satisfying Hypothesis \ref{hypothese_suffisante}. Then, every $f_\alpha$ verifies modified scattering. More precisely there exists $\widetilde{f}_{\alpha\infty}\in C_c^0(\R^3_x\times\R^3_v)$ and $\E,\mathbb{B}\in C^0(\R^3_v,\R^3_v)$ such that, for any $\alpha$ and all $t\geq 1,\, (x,v)\in\R^3_x\times\R^3_v$ verifying $\displaystyle tv^0_\alpha-\log(t)\frac{e_\alpha}{v^0_\alpha}\E(v_\alpha)\cdot \widehat{v_\alpha}\geq 0$,
    \begin{equation*}
        \left|f_\alpha\left(tv^0_\alpha-\log(t)\frac{e_\alpha}{v^0_\alpha}\E(v_\alpha)\cdot\widehat{v_\alpha},x+tv-\log(t)\frac{e_\alpha}{v^0_\alpha}\big(\E(v_\alpha)+\widehat{v_\alpha}\times \mathbb{B}(v_\alpha)\big),v\right)-\widetilde{f}_{\alpha\infty}(x,v)\right|\lesssim \frac{\log^6(2+t)}{2+t}.
    \end{equation*}
\end{theorem}

\begin{remark}
    Unlike \cite{Bigorgne_2023,Pankavich_BenArtzi_2024} we modify the characteristics of the operator $v^0_\alpha\partial_t +v\cdot\nabla_x$ instead of $\partial_t +\widehat{v_\alpha}\cdot\nabla_x$. This is more consistent with the Lorentz invariance of \eqref{equation_RVM1} that we will exploit in a forthcoming article \cite{breton_non_completeness_2025} (see Remark \ref{remarque_prochain_article}).\footnote{To observe the Lorentz invariance of the Vlasov-Maxwell system, one has to write the Vlasov equation as in \eqref{equation_RVM1} rather than in \eqref{equation_RVM_simplifiee_f_alpha} below.} However, taking 
    \begin{equation*}
        \widetilde{\mathcal{C}}_v:=\widehat{v}\big(\E(v)\cdot\widehat{v}\big)-\big(\E(v)+\widehat{v}\times \mathbb{B}(v)\big),
    \end{equation*}
    we obtain from Theorem \ref{theoreme_principal_scattering_modifie} the same statement as \cite{Bigorgne_2023,Pankavich_BenArtzi_2024}
    \begin{equation*}
        \left|f_\alpha\left(t,x+t\widehat{v_\alpha}+\log(t)\frac{e_\alpha}{v^0_\alpha}\widetilde{\mathcal{C}}_{v_\alpha},v\right)-f_{\alpha\infty}(x,v)\right|\lesssim \frac{\log^6(2+t)}{2+t},\qquad f_{\alpha\infty}(x,v)=\widetilde{f}_{\alpha\infty}\left(x+\frac{e_\alpha}{v^0_\alpha}\log(v^0_\alpha)\widetilde{\mathcal{C}}_{v_\alpha},v\right).
    \end{equation*}
    In fact, these two formulations can be derived from each other and it will be more convenient to prove the latter one.
\end{remark}
\begin{remark}
    \label{remarque_prochain_article}
    In a forthcoming article \cite{breton_non_completeness_2025}, we will show that linear scattering, that is $g_\alpha(t,\cdot)\xrightarrow[t\rightarrow\infty]{L^1_{x,v}}g_{\alpha\infty}$, is a non-generic phenomenon. More precisely, the subset of the initial data leading to linear scattering constitutes a codimension 1 submanifold.
\end{remark}


\subsection{Ideas of the proof}
We detail here the arguments used to prove Theorem \ref{theoreme_principal_scattering_modifie}. Let $(f_\alpha,E,B)$ be a solution to \eqref{equation_RVM1} that satisfies Hypothesis \ref{hypothese_suffisante}. For the sake of presentation, here we assume that $m_\alpha=1$ so that $v=v_\alpha$. 
We begin by composing $f_\alpha$ by the linear flow to consider $g_\alpha(t,x,v)=f_\alpha(t,x+t\widehat{v},v)$. We then derive two key properties 
    $$\supp g_\alpha(t,\cdot)\subset\{(x,v)\,|\,|x|\lesssim \log(2+t),\,|v|\leq \beta\},$$
    $$|\nabla_x g_\alpha(t,\cdot)|\lesssim 1,\quad |\nabla_v g_\alpha(t,\cdot)|\lesssim \log^2(2+t).$$

Moreover, the spatial support of $f_\alpha(t,\cdot)$ is included in $\{|x|\leq \gamma t +k\}$, where $\gamma<1$. Consequently, the Lorentz force $L(t,x,v):=E(t,x)+\widehat{v}\times B(t,x)$ satisfies, on the support of $f_\alpha$,
\begin{equation}
    |L(t,x,v)|\lesssim (2+t)^{-2}.
\end{equation}
The first idea is to look for linear scattering, in which case we would have that $g_\alpha(t,\cdot)$ converges as $t\rightarrow \infty$. For this matter, we compute 
\begin{align*}
    \partial_t g_\alpha(t,x,v)=&-e_\alpha(E(t,x+t\widehat{v})+\widehat{v}\times B(t,x+t\widehat{v}))\cdot (\nabla_v f_\alpha)(t,x+t\widehat{v},v)\\
    =&\,e_\alpha \frac{t}{v^0}\Big[L(t,x+t\widehat{v},v)-\left(L(t,x+t\widehat{v},v)\cdot \widehat{v}\right)\widehat{v}\Big]\cdot\nabla_x g_\alpha(t,x,v) + O(\log^2(2+t)(2+t)^{-2}).
\end{align*}

With our estimates, we can merely control the first term by $t^{-1}$, preventing us from concluding that $g_\alpha(t,\cdot)$ converges as $t\rightarrow+\infty$. Note that this is consistent with Remark \ref{remarque_prochain_article}. However, by further investigating the asymptotic behavior of $(E,B)$, we can still expect to prove that $f_\alpha$ converges along modifications of the linear characteristics. To achieve this, we are led to determine the leading order contribution of the source terms $\rho$ and $j$ in the Maxwell equations. 
For the linearized system, the asymptotic behavior of $\rho$ and $j$ is governed by the spatial average $\int f_\alpha\mathrm{d}x$, which, in this setting, is a conserved quantity. To this end,  we begin by proving the existence of the asymptotic charge $Q^\alpha_\infty$ such that $\int f_\alpha\mathrm{d}x$ converges to $Q_\infty^\alpha$.
Now, let $Q_\infty(v)=\sum_\alpha e_\alpha Q^\alpha_{\infty}(v)$. This allows us to consider the asymptotic charge and the asymptotic current densities
\begin{equation}
    \rho^{as}(t,x):=\frac{1}{t^3}\left[\langle \cdot\rangle^5Q_\infty\right]\Big(\widecheck{x/t}\Big)\1_{|x|<t},\quad j^{as}(t,x):=\frac{x}{t^4}\left[\langle \cdot\rangle^5Q_\infty\right]\Big(\widecheck{x/t}\Big)\1_{|x|<t}, 
\end{equation}
where $u\mapsto\widecheck{u}$ is the inverse of the relativistic speed $u\mapsto \widehat{u}$. These densities satisfy 
\begin{equation*}
    |\rho(t,x)-\rho^{as}(t,x)|+|j(t,x)-j^{as}(t,x)|\lesssim \frac{\log^6(2+t)}{2+t}.
\end{equation*}

The previous arguments allow us to define $\E,\mathbb{B}\in C^0(\R^3_v,\R^3_v)$, which verify
\begin{equation}
    E(t,x+t\widehat{v})=\frac{1}{t^2}\E(v)+ O\left(\frac{\log^6(2+t)}{(2+t)^3}\right),\qquad  B(t,x+t\widehat{v})=\frac{1}{t^2}\mathbb{B}(v)+ O\left(\frac{\log^6(2+t)}{(2+t)^3}\right).
\end{equation}

\begin{remark}
    Although it is not immediately apparent in Sections \ref{sous-section_estimee_champs1}--\ref{sous-section_estimee_champs2}, it turns out that $t^{-2}\E\left(\widecheck{\frac{x}{t}}\right)$ and $t^{-2}\mathbb{B}\left(\widecheck{\frac{x}{t}}\right)$ can be interpreted as follows. Let $E^{as}$ and $B^{as}$ be the solutions to
    \begin{equation*}
        \square E^{as}=-\nabla_x \rho^{as}-\partial_t j^{as},\quad \square B^{as}=\nabla_x \times j^{as},
    \end{equation*}
    with trivial data at $t_0>0$. Then for $t\geq T(t_0)$ large enough, we have for all $|x|\leq \gamma t +k$
    \begin{equation*}
        E^{as}(t,x)=t^{-2}\E\left(\widecheck{\frac{x}{t}}\right),\quad B^{as}(t,x)=t^{-2}\mathbb{B}\left(\widecheck{\frac{x}{t}}\right).
    \end{equation*}
    Moreover, for $t$ large enough, we have $t^{-2}\E\left(\widecheck{\frac{x}{t}}\right)=E_T^{as}$ in the Glassey-Strauss decomposition of the electromagnetic field $(E^{as},B^{as})$, associated to the singular distribution function $f_\alpha^{as}(t,x,v)=\delta(x-t\widehat{v})Q_\infty^\alpha(v)$ through Proposition \ref{proposition_decomp_glassey_strauss}. We refer to \cite[Section 5]{bigorgne_2023_scatteringmap} for more information.
\end{remark}
To conclude, it remains to prove the modified scattering statement for the density functions $f_\alpha$. Using the above estimates for the fields, we derive 
\begin{equation}
    \partial_t\left(f_\alpha\left(t,x+t\widehat{v},v\right)\right)=e_\alpha\left[-\left(\frac{\mathbb{L}(v)}{tv^0}\cdot \widehat{v}\right)\widehat{v}+\frac{\mathbb{L}(v)}{tv^0}\right]\cdot\nabla_x f(t,x+t\widehat{v},v) + O\left(\frac{\log^6(2+t)}{(2+t)^2}\right),
\end{equation}
where $\mathbb{L}(v)=\E(v)+\widehat{v}\times\mathbb{B}(v)$. We finally introduce the correction 
    \begin{equation}
        \widetilde{\mathcal{C}}_v:=\widehat{v}\big(\E(v)\cdot\widehat{v}\big)-\big(\E(v)+\widehat{v}\times \mathbb{B}(v)\big),
    \end{equation}
which, once multiplied by $\frac{1}{v^0}\log(t)$, measures how much the characteristics of the Vlasov operator deviate from the linear ones. It allows us to obtain the modified scattering statement and that the asymptotic state $f_{\alpha\infty}$ is compactly supported.\\

\subsection{Structure of the paper}

Section \ref{section_preliminary_results} contains several statements that are needed in our proof of Theorem \ref{theoreme_principal_scattering_modifie}. We first introduce the function $g_\alpha$ by composing $f_\alpha$ with the linear flow. Then we compute the support of $f_\alpha$ and two key properties on $g_\alpha$ and its first order derivatives. We end this section by studying the asymptotic properties of $\int_{\R^3_x} f_\alpha \mathrm{d}x$ and $\int_{\R^3_v} f_\alpha \mathrm{d}v$. Finally, in section \ref{section_modified_scattering}, we further investigate the asymptotic behavior of $(E,B)$ and show that the densities $f_\alpha$ exhibit a modified scattering dynamic. Then, we prove the compactness of the support of $\widetilde{f}_{\alpha\infty}$ (see Proposition \ref{proposition_h_support_compact}), concluding the proof of Theorem \ref{theoreme_principal_scattering_modifie}.


\section{Preliminary results}
In the following sections, we consider $(f_\alpha,E,B)$ a solution to \eqref{equation_RVM1} that satisfies Hypothesis \ref{hypothese_suffisante}.
\label{section_preliminary_results}
We begin this section by giving a lemma about the inverse of $v\mapsto\widehat{v}$.
\begin{lemma}
    \label{lemme_jacobien_hat_v}
    We define on $\{u\in\R^3,\, |u|<1\}$ the map ~$\widecheck{}$~ by 
    \begin{equation}
        u\mapsto \widecheck{u}:=\frac{u}{1-|u|^2}.
    \end{equation}
    In particular 
    \begin{equation*}
        \forall |u|<1,\quad \forall v\in\R^3_v,\quad \widehat{\widecheck{u}}=u,\quad \widecheck{\widehat{v}}=v.
    \end{equation*}
    Finally, the Jacobian determinant of $v\in\R^3\mapsto \widehat{v}$ is $\langle v\rangle^{-5}$.
\end{lemma}

Let us begin by introducing, to simplify the notations, 
\begin{equation}
    f^\alpha(t,x,v)=f_\alpha(t,x,m_\alpha v).
\end{equation}
Notice here that the support of ${f^\alpha}_0$ is now included in $\{(x,v)\,|\, |x|\leq k, |v|\leq k_\alpha\}$ with $k_\alpha=\frac{k}{m_\alpha}$. Moreover, $f^\alpha$ satisfies the following equation 
\begin{equation}
    \label{equation_RVM_simplifiee_f_alpha}
    \partial_t f^\alpha+\widehat{v}\cdot\nabla_x f^\alpha +\frac{e_\alpha}{m_\alpha}(E+\widehat{v}\times B)\cdot\nabla_v f^\alpha=0.
\end{equation}
From this we can introduce 
\begin{equation*}
    g^\alpha(t,x,v)=f^\alpha(t,x+t\widehat{v},v),\quad g_\alpha(t,x,v)=f_\alpha(t,x+t\widehat{v_\alpha},v).
\end{equation*}
Notice that this notation is consistent since $g^\alpha(t,x,v)=g_\alpha(t,x,m_\alpha v)$. Moreover, the derivatives of $g^\alpha$ can be expressed as the following 
\begin{align}
    \nabla_x g^\alpha(t,x,v)&=(\nabla_x f^\alpha)(t,x+t\widehat{v},v),\label{equation_nabla_x_g^alpha}\\
    \nabla_v g^\alpha(t,x,v)&= \frac{t}{v^0}\Big[(\nabla_x f^\alpha)(t,x+t\widehat{v},v)-\widehat{v}\big((\nabla_x f^\alpha)(t,x+t\widehat{v},v)\cdot\widehat{v}\big)\Big]+(\nabla_v f^\alpha)(t,x+t\widehat{v},v).\label{equation_nabla_v_g^alpha}
\end{align}
This means that for $L(t,x,v)=E(t,x)+\widehat{v}\times B(t,x)$, $g^\alpha$ will satisfy the following equation 
\begin{equation}
    \label{equation_satisfaite_par_g^alpha}
    \partial_t g^\alpha(t,x,v) +\frac{t}{v^0}\frac{e_\alpha}{m_\alpha}\Big[\widehat{v}(L(t,x+t\widehat{v},v)\cdot\widehat{v})-L(t,x+t\widehat{v},v)\Big]\cdot\nabla_x g^\alpha(t,x,v)+\frac{e_\alpha}{m_\alpha}L(t,x+t\widehat{v},v)\cdot\nabla_v g^\alpha(t,x,v)=0.
\end{equation}
\subsection{Characteristics}
We begin by introducing the characteristics of $g^\alpha$. Let $\mathcal{X}(s)=\mathcal{X}(s,t,x,v),\mathcal{V}(s)=\mathcal{V}(s,t,x,v)$ be defined by the ODE

\begin{equation}
    \label{ODE_characteristics_g}
    \left\{\begin{array}{ll}
        \dot{\mathcal{X}}(s)&=\dfrac{s}{\mathcal{V}^0(s)}\dfrac{e_\alpha}{m_\alpha}\left[\widehat{\mathcal{V}}(s)\left(L(s,\mathcal{X}(s)+s\widehat{\mathcal{V}}(s),\mathcal{V}(s))\cdot\widehat{\mathcal{V}}(s)\right)-L(s,\mathcal{X}(s)+s\widehat{\mathcal{V}}(s),\mathcal{V}(s))\right],\vspace{5px}\\
        \dot{\mathcal{V}}(s)&=\dfrac{e_\alpha}{m_\alpha}L(s,\mathcal{X}(s)+s\widehat{\mathcal{V}}(s),\mathcal{V}(s))
    \end{array}\right.
\end{equation}
\begin{remark}
    One can easily switch between the characteristics of $f^\alpha$ and $g^\alpha$. In fact, taking $X(s)=\mathcal{X}(s)+s\widehat{\mathcal{V}}(s),\, V(s)=\mathcal{V}(s)$, we derive the following ODE 
    \begin{align*}
        \dot{X}(s)&=\widehat{V}(s),\\
        \dot{V}(s)&=\frac{e_\alpha}{m_\alpha} L(s,X(s),V(s)).
    \end{align*}
    Meaning that $(X,V)$ are the characteristics of $f^\alpha$ starting from $X(t)=x+t\widehat{v},V(t)=v$.
\end{remark}
We will now show that the support of $f^\alpha(t,x,\cdot)$ is bounded, uniformly in $(t,x)$. From this we derive that the space support of $f(t,\cdot)$ is bounded by $\widehat{\beta_\alpha}t+k$. Here, contrary to \cite{Glassey_Strauss_1987}, we do not require smallness of the initial data to prove the result.
\begin{lemma}
    \label{lemme_support_f^alpha_en_x}
    There exists a constant $\beta>0$ such that, for all $t\geq 0$ and any $\alpha$
    \begin{equation}
        \supp(f^\alpha(t,\cdot))\subset \{|x|\leq\widehat{\beta_\alpha}t+k,|v|\leq \beta_\alpha\},
    \end{equation}
    where $\beta_\alpha:=\frac{\beta}{m_\alpha}$. Moreover, if $g^\alpha(t,x,v)\neq 0$ then $\forall s\geq 0$
    \begin{equation*}
        s-\left|\mathcal{X}(s)+s\widehat{\mathcal{V}}(s)\right|+2k\geq k +s(1-\widehat{\beta}_{max}),
    \end{equation*}
    where $\beta_{max}:=\sup_{1\leq\alpha\leq N} \beta_\alpha$. In particular, for $(x,v)$ in the support of $g^\alpha(t,\cdot)$, we have
    \begin{equation*}
        t-|x+t\widehat{v}|+2k\geq t(1-\widehat{\beta}_{max}).
    \end{equation*}
\end{lemma}
\begin{proof}
    The proof follows \cite[Lemma 2.1]{Wei_Yang_21}. Let $(t,x,v)$ such that $f^\alpha(t,x+t\hat{v},v)\neq 0$. This implies ${f^\alpha}_0(\mathcal{X}(0),\mathcal{V}(0))\neq 0$ and thus we are working with $|\mathcal{X}(0)|\leq k,\,|\mathcal{V}(0)|\leq k_\alpha$. We begin by introducing 
    \begin{equation*}
        U(t)=\sup\{|\mathcal{V}(s)|\,|\, 0\leq s \leq t\}.
    \end{equation*}
    Using the ODE satisfied by $X(s)=\mathcal{X}(s)+s\widehat{\mathcal{V}}(s)$, one finds 
    \begin{equation*}
        s-|\mathcal{X}(s)+s\widehat{\mathcal{V}}(s)|+2k\geq s(1-\widehat{U}(t))+k\geq s(1-\widehat{U}(t)).
    \end{equation*}
    Here we used that $\lambda\in\R_+\mapsto \frac{\lambda}{\langle \lambda\rangle}$ is increasing. Now consider $t_1,t_2\in [0,t]$. Using \eqref{equation_estimee_EB_Glassey} we derive, 
    \begin{align*}
        |\mathcal{V}(t_1)-\mathcal{V}(t_2)|&\leq \int_0^t \frac{C}{(s-|\mathcal{X}(s)+s\widehat{\mathcal{V}}(s)|+2k)(s+|\mathcal{X}(s)+s\widehat{\mathcal{V}}(s)|+2k)}\mathrm{d}s\\
        &\leq \int_0^{k/(1-\widehat{U}(t))}\frac{C}{(s+2k)k}\mathrm{d}s+\int_{k/(1-\widehat{U}(t))}^{+\infty}\frac{C}{(s(1-\widehat{U}(t))+k)(s+2k)}\mathrm{d}s.
    \end{align*}
    Computing the last two integrals we derive 
    \begin{equation*}
        |\mathcal{V}(t_1)-\mathcal{V}(t_2)|\leq \frac{C}{k}\left[\log\left(\frac{3}{2}\right)-\log\left(1-\widehat{U}(t)\right)+\frac{1}{2}\phi_0\left(\frac{1}{2}-\widehat{U}(t)\right)\right],
    \end{equation*}
    where 
    \begin{equation*}
        \phi_0(z)=\frac{\ln(1+z)}{z},\quad \phi(0)=1, \quad -1<z<+\infty.
    \end{equation*}
    From this we follow \cite[Lemma 2.1]{Wei_Yang_21} and find a constant $\widetilde{C}$ independent of $t$ such that 
    \begin{equation*}
        |\mathcal{V}(t)|\leq 2|V(0)|+\widetilde{C}\leq 2k_\alpha +\widetilde{C}.
    \end{equation*}
    Hence, there exists $\beta>0$ such that, for $\beta_\alpha=\frac{\beta}{m_\alpha}$, 
    \begin{equation*}
        |v|\leq \beta_\alpha.
    \end{equation*}
    It remains to prove the estimate on the support of $f^\alpha(t,\cdot,v)$. Since $|v|\leq \beta_\alpha$ and $g^\alpha$ is constant along its characteristics, we know that $\left|\mathcal{V}(s)\right|\leq \beta_\alpha$. So we derive directly 
    \begin{equation*}
        \left|\mathcal{X}(s)+s\widehat{\mathcal{V}}(s)\right|\leq |\mathcal{X}(0)|+\widehat{\beta_\alpha}s\leq k+\widehat{\beta_\alpha}s.
    \end{equation*}
    Implying directly $|x+t\widehat{v}|\leq k+\widehat{\beta_\alpha}t$. This gives the inclusion for the support of $f^\alpha$ as well as the other two inequalities.
\end{proof}

\begin{remark}
    One can easily go back to $f_\alpha$ to find its support. In fact, we have
    \begin{equation*}
        \supp f_\alpha(t,\cdot)\subset \{(x,v)\in\R^3_x\times\R^3_v\,|\, |x|\leq \widehat{\beta}_{max}t+k,\,|v|\leq \beta\}.
    \end{equation*}
    Moreover, since $(f_\alpha,E,B)$ is a solution to \eqref{equation_RVM1} with compactly supported initial data,
    \begin{equation*}
        \supp (E,B)(t,\cdot)\subset \{x\in\R^3_x\,|\, |x|\leq t+k\}.
    \end{equation*}
\end{remark}


\subsection{Properties of $g^\alpha$}
With these properties for the characteristics of $g^\alpha$ in mind, we now want to estimate the support of $g^\alpha$ and control its derivatives. 
\begin{proposition}
    \label{proposition_support_g}
    There exists a constant $C>0$ such that, for all $t\geq 0$ and any $\alpha$,
    \begin{equation}
        \supp(g^\alpha(t,\cdot))\subset \{(x,v)\in\R^3_x\times\R^3_v\,|\, |x|\leq C\log(2+t),\, |v|\leq \beta_\alpha\}.
    \end{equation}
\end{proposition}
\begin{proof}
    With the previous notations, we have, thanks to \eqref{equation_estimee_EB_Glassey}, 
    \begin{equation*}
        \big|\dot{\mathcal{X}}(s)\big|\lesssim \frac{s}{(s+|\mathcal{X}(s)+s\widehat{\mathcal{V}}(s)|+2k)(s-|\mathcal{X}(s)+s\widehat{\mathcal{V}}(s)|+2k)}. 
    \end{equation*}
    Now, thanks to Proposition \ref{lemme_support_f^alpha_en_x}, we have $\big|\dot{\mathcal{X}}(s)\big|\lesssim \frac{1}{s+2}$, which implies
    \begin{equation*}
        |\mathcal{X}(s)|\lesssim k+\log(2+s)\lesssim \log(2+s).
    \end{equation*}
\end{proof}
We now estimate $g^\alpha$ and its first order derivatives.

\begin{proposition}
    \label{proposition_estimee_gradv_g}
    Consider $(t,x,v)$ with $t\geq 0$. We have the following estimates, for any $1\leq\alpha\leq N$,
    \begin{align}
        |g^\alpha(t,x,v)| &\leq \|{f^\alpha}_0\|_\infty,\\
        |\nabla_xg^\alpha(t,x,v)|& \lesssim 1,\\
        |\nabla_vg^\alpha(t,x,v)|& \lesssim \log^2(2+t).
    \end{align}
\end{proposition}

\begin{proof}
    The first estimate is immediate by using the characteristics. Then, recall \eqref{equation_satisfaite_par_g^alpha} and let $\mathcal{L}$ be the associated operator such that $\mathcal{L}g=0$. We have 
    \begin{align*}
        \mathcal{L}(\partial_{x_i}g^\alpha)=&-\frac{t}{v^0}\frac{e_\alpha}{m_\alpha}\partial_{x_i}\left[\widehat{v}(L(t,x+t\widehat{v},v)\cdot\widehat{v})-L(t,x+t\widehat{v},v)\right]\cdot\nabla_xg^\alpha-\frac{e_\alpha}{m_\alpha}\partial_{x_i}[L(t,x+t\widehat{v},v)]\cdot\nabla_vg^\alpha,\\
        \mathcal{L}(\partial_{v_i}g^\alpha)=&-t\frac{e_\alpha}{m_\alpha}\partial_{v_i}\left[\frac{1}{v^0}\Big(\widehat{v}(L(t,x+t\widehat{v},v)\cdot\widehat{v})-L(t,x+t\widehat{v},v)\Big)\right]\cdot\nabla_xg^\alpha-\frac{e_\alpha}{m_\alpha}\partial_{v_i}[L(t,x+t\widehat{v},v)]\cdot\nabla_vg^\alpha.
    \end{align*}
    Now let us consider $\mathcal{X}(s)=\mathcal{X}(s,t,x,v),\mathcal{V}(s)=\mathcal{V}(s,t,x,v)$ the characteristics of $g^\alpha$. Recall Lemma \ref{lemme_support_f^alpha_en_x}, so for $g^\alpha(t,x,v)\neq 0$, we have
    \begin{equation*}
        s-|\mathcal{X}(s)+s\widehat{\mathcal{V}}(s)|+2k\geq s(1-\widehat{\beta}_{max})+k.
    \end{equation*}
    This implies the following estimates 
    \begin{equation*}
        |(E,B)(s,\mathcal{X}(s)+s\widehat{\mathcal{V}}(s))|\lesssim \frac{1}{(2+s)^2},\qquad |\nabla_x(E,B)(s,\mathcal{X}(s)+s\widehat{\mathcal{V}}(s))|\lesssim \frac{\log(2+s)}{(2+s)^3}.
    \end{equation*}
    We can now use the equation satisfied by $\partial_{x_i}g^\alpha$ and $\partial_{v_i}g^\alpha$ to derive, thanks to Duhamel's principle,

    \begin{align*}
        \left|(\nabla_x g^\alpha)(\tau,\mathcal{X}(\tau),\mathcal{V}(\tau))\right|&\lesssim \|f_0\|_{C^1}+\int_0^\tau \frac{\log(2+s)}{(2+s)^2}|\nabla_xg^\alpha|+\frac{\log(2+s)}{(2+s)^3}|\nabla_vg^\alpha|\mathrm{d}s,\\
        \left|(\nabla_v g^\alpha)(\tau,\mathcal{X}(\tau),\mathcal{V}(\tau))\right|&\lesssim \|f_0\|_{C^1}+\int_0^\tau \frac{\log(2+s)}{(2+s)}|\nabla_xg^\alpha|+\frac{\log(2+s)}{(2+s)^2}|\nabla_vg^\alpha|\mathrm{d}s,
    \end{align*}
    where in the integral $\nabla_xg^\alpha,\nabla_v g^\alpha$ are evaluated at $(s,\mathcal{X}(s),\mathcal{V}(s))$. Since $s\mapsto \frac{\log(2+s)}{(2+s)^2}$ is integrable, by Grönwall's inequality we have for all $\tau\geq 0$
    \begin{align}
        \left|(\nabla_x g^\alpha)(\tau,\mathcal{X}(\tau),\mathcal{V}(\tau))\right|&\lesssim \|f_0\|_{C^1} +\int_0^\tau\frac{\log(2+s)}{(2+s)^3}|\nabla_vg^\alpha|\mathrm{d}s,\label{equation_nabla_x_g_caracteristiques}\\
        \left|(\nabla_v g^\alpha)(\tau,\mathcal{X}(\tau),\mathcal{V}(\tau))\right|&\lesssim \|f_0\|_{C^1}+\int_0^\tau \frac{\log(2+s)}{(2+s)}|\nabla_xg^\alpha|\mathrm{d}s.\label{equation_nabla_v_g_caracteristiques}
    \end{align}
    We now insert \eqref{equation_nabla_x_g_caracteristiques} in \eqref{equation_nabla_v_g_caracteristiques} to derive,
    \begin{align*}
        \left|(\nabla_v g^\alpha)(\tau,\mathcal{X}(\tau),\mathcal{V}(\tau))\right|&\lesssim \|f_0\|_{C^1}\left(1+\int_0^\tau \frac{\log(2+s)}{(2+s)}\mathrm{d}s\right) +\int_0^\tau \frac{\log(2+s)}{(2+s)}\left(\int_0^s\frac{\log(2+u)}{(2+u)^3}|\nabla_v g^\alpha|\mathrm{d}u\right)\mathrm{d}s\\
        &\lesssim\log^2(2+\tau)+\log^2(2+\tau)\int_0^\tau \frac{\log(2+u)}{(2+u)^3}|\nabla_v g^\alpha|\mathrm{d}u.
    \end{align*}
    We now apply Gronwall's inequality to $G(s)=|\nabla_v g^\alpha(s,\mathcal{X}(s),\mathcal{V}(s))|\log^{-2}(2+s)$. Since $s\mapsto \frac{\log^3(2+s)}{(2+s)^3}$ is integrable we derive
    \begin{equation*}
        \frac{1}{\log^2(2+\tau)}\left|(\nabla_v g^\alpha)(\tau,\mathcal{X}(\tau),\mathcal{V}(\tau))\right|=G(\tau)\lesssim 1,
    \end{equation*}
    and the estimate on $\nabla_v g^\alpha$ follows. Inserting the estimate on $\nabla_v g^\alpha$ in \eqref{equation_nabla_x_g_caracteristiques} we derive the other estimate. Finally, taking $\tau=t$, we derive the result.
\end{proof}

\begin{remark}
    From \eqref{equation_nabla_x_g^alpha}--\eqref{equation_nabla_v_g^alpha} and the above proposition, one can easily derive estimates on the derivatives of $f^\alpha$. Moreover, the derivatives of $f_\alpha$ (resp. $g_\alpha$) satisfy the same estimates as those satisfied by $f^\alpha$ (resp. $g^\alpha$).
\end{remark}

\subsection{Convergence of the spatial average}
We focus on the spatial average of $g^\alpha$ since this quantity governs the asymptotic behavior of the source terms in the Maxwell equations.

\begin{proposition}
    \label{corollaire_estimee_Q_infini}
    For any $\alpha$, there exists $Q^\alpha_\infty\in C^0_c(\R^3_v)$ such that, for all $t\geq 0$ and $v\in\R^3_v$, 
    \begin{equation}
        \label{equation_corollaire_estimee_Q_infini}
        \left|\int_{\R^3_x}g^\alpha(t,x,v)\mathrm{d}x - Q^\alpha_\infty(v)\right|\lesssim \frac{\log^5(2+t)}{2+t}.
    \end{equation}
\end{proposition}

\begin{proof} 
We begin by integrating \eqref{equation_satisfaite_par_g^alpha} over $\R^3_x$ to derive
\begin{align*}
    \partial_t \int_{\R^3_x}g^\alpha(t,x,v)\mathrm{d}x =&-\frac{t}{v^0}\frac{e_\alpha}{m_\alpha}\int_{\R^3_x}\Big[\widehat{v}(L(t,x+t\widehat{v},v)\cdot\widehat{v})-L(t,x+t\widehat{v},v)\Big]\cdot\nabla_x g^\alpha(t,x,v)\mathrm{d}x\\
    &-\frac{e_\alpha}{m_\alpha}\int_{\R^3_x}L(t,x+t\widehat{v},v)\cdot\nabla_v g^\alpha(t,x,v)\mathrm{d}x.
\end{align*}
The second term of the right-hand side can directly be dealt with. Indeed, thanks to \eqref{equation_estimee_EB_Glassey}, Lemma \ref{lemme_support_f^alpha_en_x} and Propositions  \ref{proposition_support_g}--\ref{proposition_estimee_gradv_g}, we have 
\begin{equation*}
    \left|\frac{e_\alpha}{m_\alpha}\int_{\R^3_x}L(t,x+t\widehat{v},v)\cdot\nabla_v g^\alpha(t,x,v)\mathrm{d}x\right|\lesssim \int_{|x|\lesssim \log(2+t)} \frac{\log^2(2+t)}{(2+t)^2} \mathrm{d}x\lesssim \frac{\log^5(2+t)}{(2+t)^2}.
\end{equation*}
It remains to study the first term. By integration by parts, we derive
\begin{align*}
    t\int_{\R^3_x}\Big[\widehat{v}(L(t,x+t\widehat{v},v)\cdot\widehat{v})-L(t,x+t\widehat{v},v)\Big]\cdot\nabla_x g^\alpha(t,x,v)\mathrm{d}x&=t\int_{\R^3_x} (\nabla_x\cdot L)(t,x+t\hat{v},v) g^\alpha(t,x,v)\mathrm{d}x \\
    &-t\int_{\R^3_x} \sum_{i=1}^3 \widehat{v}^i((\partial_{x_i} L)(t,x+t\hat{v},v)\cdot \widehat{v})g^\alpha(t,x,v)\mathrm{d}x.
\end{align*}
Again from \eqref{equation_estimee_gradEB_Glassey}, Lemma \ref{lemme_support_f^alpha_en_x} and Propositions \ref{proposition_support_g}--\ref{proposition_estimee_gradv_g}, we obtain 
\begin{align*}
    t\left|\int_{\R^3_x}\Big[\widehat{v}(L(t,x+t\widehat{v},v)\cdot\widehat{v})-L(t,x+t\widehat{v},v)\Big]\cdot\nabla_x g^\alpha(t,x,v)\mathrm{d}x\right|&\lesssim t\int_{\R^3_x} |\nabla_x(E,B)(t,x+t\widehat{v})|\,|g^\alpha(t,x,v)|\mathrm{d}x\\
    &\lesssim \int_{|x|\lesssim \log(2+t)} \frac{\log(2+t)}{(2+t)^2} \mathrm{d}x\\
    &\lesssim \frac{\log^4(2+t)}{(2+t)^2}.
\end{align*}

Finally, combining these two estimates, we obtain 
\begin{equation*}
    \left|\partial_t \int_{\R^3_x}g^\alpha(t,x,v)\mathrm{d}x\right|\lesssim \frac{\log^5(2+t)}{(2+t)^2}.
\end{equation*}
Now, since $\partial_t \int_{\R^3_x} g^\alpha(t,x,v)\mathrm{d}x$ is integrable in $t$, this proves the existence of the limit $Q^\alpha_\infty$. Moreover $g^\alpha(t,x,\cdot)$ has its support in $\{v\in\R^3_v,\,|v|\leq \beta_\alpha\}$ so we already know that $\supp(Q^\alpha_\infty)\subset \{v\in\R_v^3,\,|v|\leq \beta_\alpha\}$. \\
Finally, since $\int_{\R^3_x}f(t,x,\cdot)\mathrm{d}x$ is continuous and converges uniformly towards $Q^\alpha_\infty$ we know that $Q^\alpha_\infty$ is also continuous. 
\end{proof}


\subsection{Link between the particle density and the asymptotic charge}
In \cite{Glassey_Strauss_1987} they prove that the velocity average decays like $(1+t)^{-3}$. Here we provide the asymptotic expansion of $\int f_\alpha \mathrm{d}v$. The following proposition justifies, for $h(v)=1$ and $h(v)=\widehat{v_\alpha}$, the asymptotic expansion of the charge and current densities $(\rho,j)$ stated in the outline of the proof. Recall, in particular, that $\widehat{v_\alpha}(m_\alpha v)=\widehat{v}$.

\begin{proposition}
    \label{proposition_lien_f_et_Q_infini}
    Let $h\in C^1(\R^3_v)$. Then for any $\alpha$, all $t>0$ and all $|x|<t$ we have 
    \begin{equation}
        \left|t^3\int_{\R^3_v} h(v)f_\alpha(t,x,v)\mathrm{d}v-m_\alpha^3\left[\langle \cdot\rangle^5 h(m_\alpha \cdot) Q^\alpha_{\infty}\right]\left(\widecheck{\frac{x}{t}}\right)\right|\lesssim \frac{\log^6(2+t)}{2+t}\sup_{|v|\leq \beta} (|h(v)|+|\nabla_v h(v)|).
    \end{equation}
\end{proposition}

\begin{proof} Since $f_\alpha(t,\cdot)$ and $Q_\infty^\alpha$ are continuous and compactly supported, it is enough to prove the estimate for $t\geq 1$. We begin by the change of variable $w=v_\alpha$ so that
\begin{equation*}
    \int_{\R^3_v} h(v)f_\alpha(t,x,v)\mathrm{d}v=m_\alpha^3\int_{\R^3_v} h(m_\alpha v)f^\alpha(t,x,v)\mathrm{d}v.
\end{equation*}
Applying Proposition \ref{corollaire_estimee_Q_infini} to $v=\widecheck{\frac{x}{t}}$, in view of the support of $g^\alpha$ and $Q^\alpha_\infty$, we derive
\begin{equation}
    \label{equation_preuve1_estimee_g^alpha}
    \left|\int_{\R^3_y}\left[\langle \cdot \rangle^5 h g^\alpha(t,y,\cdot)\right]\left(\widecheck{\frac{x}{t}}\right)\mathrm{d}y - \left[\langle \cdot \rangle^5 h Q^\alpha_\infty\right]\left(\widecheck{\frac{x}{t}}\right)\right|\lesssim \frac{\log^5(2+t)}{2+t}\sup_{|v|\leq \beta_\alpha} |h(v)|.
\end{equation}
This leaves us with proving that, for any $h\in C^1(\R^3_v)$,
\begin{equation}
    \left|t^3\int_{\R^3_v}h(v)f^\alpha(t,x,v)\mathrm{d}v-\int_{\R^3_y}\left[\langle \cdot \rangle^5 h g^\alpha(t,y,\cdot)\right]\left(\widecheck{\frac{x}{t}}\right)\mathrm{d}y \right|\lesssim \frac{\log^6(2+t)}{2+t}\sup_{|v|\leq \beta_\alpha} (|h(v)|+|\nabla_v h(v)|).
    \label{equation_estimee_Q_vitesse_1}
\end{equation}
First, use Lemma \ref{lemme_jacobien_hat_v} and the change of variables $y=x-t\widehat{v}$ to derive 
\begin{equation*}
    t^3\int_{\R^3_v}h(v)f^\alpha(t,x,v)\mathrm{d}v=t^3\int_{|v|\leq \beta_\alpha}h(v)g^\alpha(t,x-t\widehat{v},v)\mathrm{d}v=\int_{|x-y|\leq \widehat{\beta_\alpha}t}\left[\langle \cdot \rangle^5 h g^\alpha(t,y,\cdot)\right]\left(\widecheck{\frac{x-y}{t}}\right)\mathrm{d}y.
\end{equation*}
Here we observe the correct function evaluated in $\widecheck{\frac{x-y}{t}}$ instead of $\widecheck{\frac{x}{t}}$. We force the desired term to appear by writing 
\begin{align*}
    \int_{|x-y|<t}\left[\langle \cdot \rangle^5 h g^\alpha(t,y,\cdot)\right]\left(\widecheck{\frac{x-y}{t}}\right)\mathrm{d}y =&  \int_{\R^3_y}\left[\langle \cdot \rangle^5 h g^\alpha(t,y,\cdot)\right]\left(\widecheck{\frac{x}{t}}\right)\mathrm{d}y \\
    & +  \int_{|x-y|<t}\left[\langle \cdot \rangle^5 h g^\alpha(t,y,\cdot)\right]\left(\widecheck{\frac{x-y}{t}}\right)\mathrm{d}y -  \int_{|x-y|<t}\left[\langle \cdot \rangle^5 h g^\alpha(t,y,\cdot)\right]\left(\widecheck{\frac{x}{t}}\right)\mathrm{d}y\\
    &- \int_{|x-y|\geq t}\left[\langle \cdot \rangle^5 h g^\alpha(t,y,\cdot)\right]\left(\widecheck{\frac{x}{t}}\right)\mathrm{d}y\\
    =&~I_0+I_1+I_2.
\end{align*}
We now show that $I_1$ and $I_2$ are both $O(\log^6(2+t)(2+t)^{-1})$. \\
\textbf{Estimate of $I_1$.} For a fixed $y$ we want to apply the mean value theorem to $G:v\mapsto \left[\langle \cdot \rangle^5 h g^\alpha(t,y,\cdot)\right]\left(\widecheck{v}\right)$. Since $|\nabla_v \widecheck{v}|\lesssim \langle\widecheck{v}\rangle^3$, by differentiating $G$ we get 
\begin{align*}
    |\nabla_vG(v)|&\lesssim \langle \widecheck{v}\rangle^8|\nabla_v h(\widecheck{v})||g^\alpha(t,y,\widecheck{v})| + \langle\widecheck{v}\rangle^7| h(\widecheck{v})||g^\alpha(t,y,\widecheck{v})| + \langle\widecheck{v}\rangle^8|h(\widecheck{v})||\nabla_vg^\alpha(t,y,\widecheck{v})|\\
    &\lesssim \sup_{|u|\leq \beta_\alpha} \langle u\rangle^8(|g^\alpha(t,y,u)| + |\nabla_v g^\alpha(t,y,u)|)(|h(u)|+|\nabla_v h(u)|)\\
    & \lesssim \log^2(2+t)\sup_{|u|\leq \beta_\alpha} (|h(u)|+|\nabla_v h(u)|), 
\end{align*}
by Proposition \ref{proposition_estimee_gradv_g}. Now by the mean value theorem and using the support of $g^\alpha$, we get,
\begin{equation}
    |I_1|\lesssim \sup_{|v|\leq \beta_\alpha} (|h(v)|+|\nabla_v h(v)|)\int_{|y|\lesssim \log(2+t)}\frac{|y|}{t}\log^2(2+t)\mathrm{d}y\lesssim \frac{\log^6(2+t)}{2+t}\sup_{|v|\leq \beta_\alpha} (|h(v)|+|\nabla_v h(v)|).
\end{equation}
\textbf{Estimate of $I_2$.} Recall that $|x|<t$, this allows us to get, for $v=\widecheck{\frac{x}{t}}$ and $|y-x|\geq t$,
\begin{equation*}
    1=\langle v\rangle^2\left(1-\left|\frac{x}{t}\right|^2\right)\leq \frac{|y|(t+|x|)\langle v\rangle^2}{t^2}\leq 2\frac{|y|\langle v\rangle^2}{t},
\end{equation*}
implying that 
\begin{equation*}
    |I_2|\leq \frac{2}{t}\int_{|x-y|\geq t}|y|\left[\langle \cdot \rangle^7 h g^\alpha(t,y,\cdot)\right]\left(\widecheck{\frac{x}{t}}\right)\mathrm{d}y \lesssim \sup_{|v|\leq \beta_\alpha} |h(v)|\frac{2}{t}\int_{|y|\lesssim \log(t)}|y|\mathrm{d}y\lesssim \frac{\log^4(2+t)}{2+t}\sup_{|v|\leq \beta_\alpha} |h(v)|.
\end{equation*}
Combining these estimates, we get \eqref{equation_estimee_Q_vitesse_1}. With \eqref{equation_preuve1_estimee_g^alpha}, this implies the result.
\end{proof}


\section{Modified scattering}

\label{section_modified_scattering}


\subsection{Estimations of the fields}
\label{sous-section_estimee_champs1}
Let us start by recalling the Glassey-Strauss decomposition.

\begin{proposition}
    \label{proposition_decomp_glassey_strauss}
    Let $t\geq 0$ and $x\in\R^3_x$. The following decomposition of the field holds.
    \begin{equation}
        E(t,x)=E_T(t,x)+E_S(t,x)+E_{data}(t,x),
    \end{equation}
    where 
    \begin{align}
        \label{equation_ecriture_E_T} E_T(t,x)&=-\sum_{1\leq\alpha\leq N}e_\alpha\int_{|x-y|\leq t} \int_{\R^3_v}\frac{(\omega+\widehat{v_\alpha})(1-|\widehat{v_\alpha}|^2)}{(1+\widehat{v_\alpha}\cdot\omega)^2}f_\alpha(t-|x-y|,y,v)\mathrm{d}v\frac{\mathrm{d}y}{|y-x|^2}, \\
        \label{equation_ecriture_E_S} E_S(t,x)&=\sum_{1\leq\alpha\leq N} e^2_\alpha\int_{|y-x|\leq t}\int_{\R^3_v}\frac{\omega+\widehat{v_\alpha}}{1+\widehat{v_\alpha}\cdot\omega} (E+\widehat{v_\alpha}\cdot B)(t-|x-y|,y)\cdot\nabla_v f_\alpha(t-|x-y|,y,v)\mathrm{d}v\frac{\mathrm{d}y}{|x-y|}, \\
        \label{equation_ecriture_E_data} E_{data}(t,x)&=\mathcal{E}(t,x)-\sum_{1\leq\alpha\leq N}\frac{e_\alpha}{t}\int_{|y-x|=t}\int_{\R^3_v}\frac{\omega -(\widehat{v_\alpha}\cdot\omega)\widehat{v_\alpha}}{1+\widehat{v_\alpha}\cdot\omega}f_{\alpha0}(y,v)\mathrm{d}v\mathrm{d}S_y,
    \end{align}
    with
    \begin{equation*}
        \mathcal{E}(t,x)=\frac{1}{4\pi t^2}\int_{|y-x|=t}\left[E_0(y)+((y-x)\cdot\nabla)E_0(y)+t\nabla\times B_0(y)\right]\mathrm{d}S_y -\sum_{1\leq\alpha\leq N}\frac{e_\alpha}{4\pi t} \int_{|y-x|=t} \int_{\R^3_v} \widehat{v_\alpha}f_{\alpha0}(y,v) \mathrm{d}v\mathrm{d}S_y,
    \end{equation*}
    and $\omega=\frac{x-y}{|x-y|}$.
\end{proposition}
\begin{remark}
    The same decomposition holds for $B(t,x)=B_T(t,x)+B_S(t,x)+B_{data}(t,x)$. The expression $B_T$ and $B_S$ is obtained by replacing $\omega +\widehat{v_\alpha}$ by $\omega \times\widehat{v_\alpha}$ in $E_T$ and $E_S$. Moreover, the expression of $B_{data}$ only depends on the initial data, so the estimates follow similarly in the next propositions. In the following, we restrict ourselves to the study of $E$ as the analysis of $B$ is similar.
\end{remark}
As stated in the introduction, we want to identify the part of $E=E_{data}+E_S+E_T$ that decays like $t^{-2}$ for $|x|\leq \gamma t$, with $\gamma<1$. We start by showing that $E_{data}$ and $E_S$ decay at least as $t^{-3}$. For this, we have to improve the estimate obtained by Glassey and Strauss in \cite{Glassey_Strauss_1987} for $E_S$.
\begin{proposition}
    \label{proposition_estimee_E_data}
    For all $(t,x)\in\R_+\times\R^3_x$, we have the following estimate 
    \begin{equation}
        |(E_{data},B_{data})(t,x)|\lesssim \langle t\rangle^{-1}\1_{|t-|x||\leq k}.
    \end{equation}
\end{proposition}
\begin{proof} 
First, recall the expression of $E_{data}$ from \eqref{equation_ecriture_E_data}. Using the support of $f_{\alpha0},E_0,B_0$ every term of $E_{data}$ is bounded by 
\begin{equation*}
    C\left(\frac{1}{t}+\frac{1}{t^2}\right)\int_{\substack{|y-x|=t\\ |y|\leq k}} \mathrm{d}S_y\1_{|t-|x||<k},
\end{equation*}
which implies the result.
\end{proof}
\begin{proposition}
    \label{proposition_estimee_E_S}
    For all $(t,x)\in\R_+\times\R^3_x$, we have the following estimate 
    \begin{equation}
        |(E_S,B_S)(t,x)|\lesssim (t+|x|+2k)^{-1}(t-|x|+2k)^{-2}.
    \end{equation}
\end{proposition}
\begin{proof} Here we slightly refine the analysis performed for $E_S$ in \cite[Lemma 6]{Glassey_Strauss_1987} by exploiting the support of $f_\alpha$. First recall \eqref{equation_ecriture_E_S} and $\nabla_v \cdot (E+\widehat{v}\times B)=0$. By integration by parts we obtain 
\begin{equation*}
    E_S(t,x)=-\sum_{1\leq \alpha\leq N} e_\alpha^2\int_{|y-x|\leq t}\int_{\R^3_v}\nabla_v\left[\frac{\omega+\widehat{v_\alpha}}{1+\widehat{v_\alpha}\cdot\omega}\right]\cdot (E+\widehat{v_\alpha}\times B)(t-|x-y|,y)f_\alpha(t-|x-y|,y,v)\mathrm{d}v\frac{\mathrm{d}y}{|x-y|}.
\end{equation*}
Now, recall that the kernel where $\omega$ appears is bounded on the support of $f_\alpha$. Then by Proposition \ref{proposition_lien_f_et_Q_infini} we have 
\begin{align*}
    \int_{\R^3_v} |(E+\widehat{v_\alpha}\times B)(t-|x-y|,y)| f_\alpha(t-|x-y|,y,v)\mathrm{d}v&\lesssim \frac{\1_{|y|\leq \widehat{\beta}_{max}(t-|y-x|)+k}}{(t-|x-y|+|y|+2k)^{4}(t-|x-y|-|y|+2k)}\\
    &\lesssim (t-|x-y|+|y|+2k)^{-5},
\end{align*}
thanks to the support of $f_\alpha$. This implies
\begin{equation*}
    |E_S(t,x)|\lesssim \int_{|y-x|\leq t} \frac{1}{(t-|y-x|+|y|+2k)^5}\frac{\mathrm{d}y}{|x-y|}.
\end{equation*}
Now, by \cite[Lemma 7]{Glassey_Strauss_1987}, we have 
\begin{equation*}
    I:=\int_{|y-x|\leq t} \frac{1}{(t-|y-x|+|y|+2k)^5}\frac{\mathrm{d}y}{|x-y|}=\frac{1}{r}\int_0^t\int_a^b \frac{\lambda\mathrm{d}\lambda\mathrm{d}\tau}{(\tau +\lambda +2k)^{5}}\leq\frac{1}{r}\int_0^t\int_a^b \frac{\mathrm{d}\lambda\mathrm{d}\tau}{(\tau +\lambda +2k)^{4}},
\end{equation*}
where $\tau=t-|x-y|,\,\lambda=|y|$ and $r=|x|$. Moreover, the bounds of the integral are $a=|r-t+\tau|$ and $b=r+t-\tau$. We first write, as $b-a\leq b-(t-r-\tau)=2r$ and $\tau+b=t+r$,
\begin{equation*}
    I\lesssim \frac{1}{r}\int_0^t \frac{(b-a)(2\tau+b+a+4k)^2}{(\tau+a+2k)^3(\tau+b+2k)^3}\mathrm{d}\tau\lesssim \frac{1}{r}\int_0^t \frac{(b-a)\mathrm{d}\tau}{(\tau+a+2k)^3(\tau+b+2k)}\lesssim \frac{1}{t+r+2k}\int_0^t \frac{\mathrm{d}\tau}{(\tau+a+2k)^3
    }.
\end{equation*}
If $t\leq |x|< t+k$, then $I\lesssim (t+|x|+2k)^{-1} \lesssim (t+|x|+2k)^{-1}(t-|x|+2k)^{-2}$. Otherwise, $|x|<t$ and we can split $I$ in two parts 
\begin{equation*}
    I\lesssim \frac{1}{t+r+2k}\int_0^{t-r} \frac{\mathrm{d}\tau}{(t-r+2k)^3} +\frac{1}{t+r+2k}\int_{t-r}^t \frac{\mathrm{d}\tau}{(\tau+2k)^3}\lesssim \frac{1}{(t+r+2k)(t-r+2k)^2}.
\end{equation*}

\end{proof}


\subsection{Asymptotic expansion of the fields} 
\label{sous-section_estimee_champs2}

Having proved the estimate on $(E_S,E_{data})$, it remains to study $E_T$. The goal of this subsection is to find a form of asymptotic expansion for $E_T$.

\begin{proposition}
    Let $v\in\R^3_v$. Consider 
    \begin{equation}
        \E_\alpha(v):=-\int_{\substack{|y|\leq 1 \\ |y+\widehat{v}|< 1-|y|}} \left[\langle\cdot\rangle^5 W\left(\frac{y}{|y|},\cdot\right)Q^\alpha_\infty\right]\left(\frac{\widecheck{y+\widehat{v}}}{1-|y|}\right)\frac{1}{(1-|y|)^3}\frac{\mathrm{d}y}{|y|^2},
    \end{equation}
    \begin{equation}
        \mathbb{B}_\alpha(v):=-\int_{\substack{|y|\leq 1 \\ |y+\widehat{v}|< 1-|y|}} \left[\langle\cdot\rangle^5 \mathcal{W}\left(\frac{y}{|y|},\cdot\right)Q^\alpha_\infty\right]\left(\frac{\widecheck{y+\widehat{v}}}{1-|y|}\right)\frac{1}{(1-|y|)^3}\frac{\mathrm{d}y}{|y|^2},
    \end{equation}
    with $W(\omega,v)=\frac{(\omega+\widehat{v})}{\langle v\rangle^2(1+\widehat{v}\cdot\omega)^2}$ and $\mathcal{W}(\omega,v)=\frac{(\omega\times\widehat{v})}{\langle v\rangle^2(1+\widehat{v}\cdot\omega)^2}$. Then, $\E_\alpha,\mathbb{B}_\alpha\in C^0(\R^3_v)$.
\end{proposition}
\begin{proof}
     Using the support of $Q^\alpha_\infty$ we know that we integrate over $\{y|\,|y+\widehat{v}|\leq \widehat{\beta_\alpha}(1-|y|)\}$ so $|y|\leq\frac{|\widehat{v}|+\widehat{\beta_\alpha}}{1+\widehat{\beta_\alpha}}<1$ and thus $1-|y|\geq 1-\frac{|\widehat{v}|+\widehat{\beta_\alpha}}{1+\widehat{\beta_\alpha}}>0$, implying that the integral is well defined. The continuity follows as $Q_\infty^\alpha\in C^0(\R^3)$.
\end{proof}

We begin by performing the change of variables $z=\frac{y-x}{t}$, so that
\begin{equation*}
    E_T(t,x)=\sum_{1\leq\alpha\leq N} e_\alpha m_\alpha^3 E_{\alpha,T},
\end{equation*}
where 
\begin{equation*}
    E_{\alpha,T}(t,x):=-\frac{1}{t^2}\int_{|y|\leq 1 } \int_{\R^3_v}t^3 W\left(\frac{y}{|y|},v\right)f^\alpha(t(1-|y|),ty+x,v)\mathrm{d}v\frac{\mathrm{d}y}{|y|^2}.
\end{equation*}

\begin{proposition}
    \label{proposition_estimee_E_T}
    Let $0<\gamma<1$. Then, there exists $T(\gamma)\geq 0$ such that for all $t\geq T(\gamma)$ and all $|x|\leq \gamma t$ 
    \begin{equation}
       \left|t^2 E_{\alpha,T}(t,x)-\E_\alpha\left(\widecheck{\frac{x}{t}}\right)\right| + \left|t^2 B_{\alpha,T}(t,x)-\mathbb{B}_\alpha\left(\widecheck{\frac{x}{t}}\right)\right |\lesssim \frac{\log^6(2+t)}{2+t} I(\gamma),
    \end{equation}
    where $I(\gamma)$ is a constant depending on $\gamma$.
\end{proposition}
\begin{proof}
    We begin by showing that for $t$ large enough (depending on $\gamma$) and $|x|\leq \gamma t$ we have
    \begin{equation*}
        E_{\alpha,T}(t,x)=-\frac{1}{t^2}\int_{\substack{|y|\leq 1 \\ |y+\frac{x}{t}|< 1-|y|}} \int_{\R^3_v}t^3 W\left(\frac{y}{|y|},v\right)f^\alpha(t(1-|y|),ty+x,v)\mathrm{d}v\frac{\mathrm{d}y}{|y|^2}.
    \end{equation*}
    Write $t'=t(1-|y|),\,x'=t(y+\frac{x}{t})$. On the support of $f^\alpha$ we know that $|x'|\leq \widehat{\beta_\alpha}t' + k \leq \widehat{\beta}_{max}t'+k$. Thus, for $t'>\frac{k}{1-\widehat{\beta}_{max}}$, we have $t'>|x'|$. Now, on the support of $f^\alpha$ again we have 
    \begin{equation*}
        t'=t-|x'-x|\geq (1-\gamma)t-\widehat{\beta}_{max}t'-k,
    \end{equation*}
    leaving us with $t'\geq \frac{t(1-\gamma)}{1+\widehat{\beta}_{max}}-\frac{k}{1+\widehat{\beta}_{max}}$. Finally, for 
    \begin{equation*}
        t\geq T(\gamma):=\frac{2k}{(1-\gamma)(1+\widehat{\beta}_{max})}+1,
    \end{equation*}
    we derive $t'>\frac{k}{1-\widehat{\beta}_{max}}$ and thus $t'>|x'|$.\\
    Consider now 
    \begin{align}
        \label{equation_definition_mathcal_E}\mathcal{E}_T(y,t):=& t^3(1-|y|)^3\int_{\R^3_v}W\left(\frac{y}{|y|},v\right)f^\alpha(t(1-|y|),ty+x,v)\mathrm{d}v-\left[\langle\cdot\rangle^5 W\left(\frac{y}{|y|},\cdot\right)Q^\alpha_\infty\right]\left(\widecheck{\frac{y+\frac{x}{t}}{1-|y|}}\right),
    \end{align}
    so that 
    \begin{equation}
        t^2E_{\alpha,T}(t,x)-\E_\alpha\left(\widecheck{\frac{x}{t}}\right)= -\int_{\substack{|y|\leq 1 \\ |y+\frac{x}{t}|< 1-|y|}} \mathcal{E}_T(y,t) \frac{1}{(1-|y|)^3}\frac{\mathrm{d}y}{|y|^2}.
    \end{equation}
    Recall \eqref{equation_definition_mathcal_E}. Using the support of $f^\alpha$ and $Q^\alpha_\infty$, $\mathcal{E}_T(t,y)$ vanishes for $|y|> \frac{\gamma + \widehat{\beta}_{max}}{1+\widehat{\beta}_{max}}+ \frac{k}{(1+\widehat{\beta}_{max})t}$. Thus, for \mbox{$t\geq T(\gamma)\geq\frac{2k}{(1-\gamma)(1+\widehat{\beta}_{max})}$}  we have, on the support of $\mathcal{E}_T(t,\cdot)$,
    \begin{equation*}
        1-|y|\geq \frac{1}{2}\left(\frac{1-\gamma}{1+\widehat{\beta}_{max}}\right)=:K(\gamma)>0.
    \end{equation*}
    
    We now apply Proposition \ref{proposition_lien_f_et_Q_infini} with $h(v)=W(\frac{y}{|y|},v)$. Since $W\in C^1(\mathbb{S}^2\times \R^3)$ we derive 
    \begin{equation*}
         \int_{\substack{|y|\leq 1 \\ |y+\frac{x}{t}|< 1-|y|}}\left|\mathcal{E}_T(y,t)\right|\frac{1}{(1-|y|)^3} \frac{\mathrm{d}y}{|y|^2} \lesssim K(\gamma)^{-4}\int_{\substack{|y|\leq 1 \\ |y+\frac{x}{t}|< 1-|y|}}\frac{\log^6(2+t(1-|y|))}{2+t}\frac{\mathrm{d}y}{|y|^2}\lesssim\frac{\log^6(2+t)}{2+t}K(\gamma)^{-4},
    \end{equation*}
    which concludes the proof.
\end{proof}
Before giving the final estimate, we define the asymptotic fields $\E$ and $\mathbb{B}$ by
\begin{equation*}
    \E:=\sum_{1\leq\alpha\leq N} e_\alpha m_\alpha^3\E_\alpha,\qquad \mathbb{B}:=\sum_{1\leq\alpha\leq N} e_\alpha m_\alpha^3\mathbb{B}_\alpha.
\end{equation*}
\begin{corollary}
    \label{corollaire_estimee_E_E_bb}
    Let $0<\gamma<1,\, t\geq T(\gamma)$ and $|x|\leq \gamma t$. We have the following estimate
    \begin{equation}
        \left|t^2E(t,x)-\E\left(\widecheck{\frac{x}{t}}\right)\right|+\left|t^2B(t,x)-\mathbb{B}\left(\widecheck{\frac{x}{t}}\right)\right|\lesssim \frac{\log^6(2+t)}{2+t} C(\gamma).
    \end{equation}
    where $C(\gamma)$ is a constant depending only on $\gamma$ and $k$.
\end{corollary}
\begin{proof}
    This follows from the decomposition $E=E_S+E_{data}+\sum_\alpha E_{\alpha,T}$ and Propositions \ref{proposition_estimee_E_data}--\ref{proposition_estimee_E_S} and \ref{proposition_estimee_E_T}.
\end{proof}


\subsection{Proof of the modified scattering theorem}

In this subsection, we finish the proof of Theorem \ref{theoreme_principal_scattering_modifie}. First, let us detail two preliminary results.

\begin{proposition}
    \label{proposition_estimee_finale_E_E_bb}
    For any $\alpha$, all $t\geq 0,\,|v|\leq\beta_\alpha$ and all $x$ in the support of $g^\alpha(t,\cdot)$, i.e.  $|x|\leq C\log(2+t)$, we have  
    \begin{equation}
        |t^2 E(t,x+t\widehat{v})-\E(v)|+|t^2 B(t,x+t\widehat{v})-\mathbb{B}(v)|\lesssim \frac{\log^6(2+t)}{2+t}.
    \end{equation}
\end{proposition}
\begin{proof} Using Corollary \ref{corollaire_estimee_E_E_bb} we already know that for $t\geq T(\widehat{\beta}_{max})$, since $|\widehat{v}|\leq \widehat{\beta}_{max}$ 
\begin{equation*}
    |t^2 E(t,t\widehat{v})-\E(v)|\lesssim \frac{\log^6(2+t)}{2+t}C(\widehat{\beta}_{max})\lesssim \frac{\log^6(2+t)}{2+t},
\end{equation*}
where here we omit the constant $C(\widehat{\beta}_{max})$ since it only depends on $k$ and $C_0$. Now we only need to prove that 
\begin{equation}
    \label{equation_preuve2_estimee_E}
    |t^2 E(t,x+t\widehat{v})-t^2 E(t,t\widehat{v})|\lesssim \frac{\log^2(2+t)}{2+t}.
\end{equation}
We will prove the above inequality using the mean value theorem. For this we need to consider $y=\lambda (x+t\widehat{v})+(1-\lambda)t\widehat{v}\in[t\widehat{v},x+t\widehat{v}]$ with $\lambda\in[0,1]$, so that $y=\lambda x +t\widehat{v}$. Using \eqref{equation_estimee_gradEB_Glassey} we obtain 
\begin{equation*}
    |\nabla_x E(t,y)|\lesssim \frac{\log(t+|y|+2k)}{(t+|y|+2k)(t-|y|+2k)^2}.
\end{equation*}
Consider $t\geq T_1$ large enough so we have $\frac{C\log(2+t)}{t}\leq \frac{1-\widehat{\beta}_{max}}{2}$. This implies $|y|\leq \frac{1+\widehat{\beta}_{max}}{2}t$ and then $t-|y|+2k\gtrsim 2+t$ as well as $\log(t+|y|+2k)\lesssim \log(2+t)$. This grants us 
\begin{equation*}
    |\nabla_x E(t,y)|\lesssim \frac{\log(2+t)}{(2+t)^3}.
\end{equation*}
Now, the mean value theorem and $|x|\lesssim \log(2+t)$ allow us to derive \eqref{equation_preuve2_estimee_E} and obtain the result for $t\geq T_1$. It also holds on the compact interval of time $[0,T_1]$ since $E$ and $\E$ are bounded.
\end{proof}
Now, recall \eqref{equation_satisfaite_par_g^alpha}. Thanks to the estimate \eqref{equation_estimee_EB_Glassey} and Propositions \ref{proposition_estimee_gradv_g} and \ref{proposition_estimee_finale_E_E_bb}, we have 
\begin{align}
    \partial_t g^\alpha(t,x,v)=&-\frac{t}{v^0}\frac{e_\alpha}{m_\alpha}\Big[\widehat{v}(L(t,x+t\widehat{v},v)\cdot\widehat{v})-L(t,x+t\widehat{v},v)\Big]\cdot\nabla_x g^\alpha(t,x,v)+O\left(\frac{\log^2(2+t)}{(2+t)^2}\right)\nonumber\\
    =&-\frac{1}{tv^0}\frac{e_\alpha}{m_\alpha}\Big[\widehat{v}(\mathbb{L}(v)\cdot\widehat{v})-\mathbb{L}(v)\Big]\cdot\nabla_x g^\alpha(t,x,v)+O\left(\frac{\log^6(2+t)}{(2+t)^2}\right),\label{equation_finale_derivee_en_temps_g_alpha}
\end{align}
where $\mathbb{L}(v):=\E(v)+\widehat{v}\times\mathbb{B}(v)$ is the asymptotic Lorentz force. We observe that the first term on the right-hand side is of order $t^{-1}$, which is not integrable. We thus consider corrections to the linear characteristics to cancel it. We define the following corrections 
\begin{equation}
    \mathcal{D}_v:=-\widehat{v}\cdot \mathbb{L}(v),\quad \mathcal{C}_v:=-\mathbb{L}(v).
\end{equation}
Now, let us consider 
\begin{align}
    h^\alpha(t,x,v):=&f^\alpha\left(t,x+t\widehat{v}+\frac{e_\alpha}{m_\alpha v^0}\log(t)(\mathcal{C}_v-\widehat{v}\mathcal{D}_v),v~\right)~=g^\alpha\left(t,x+\frac{e_\alpha}{m_\alpha v^0}\log(t)(\mathcal{C}_v-\widehat{v}\mathcal{D}_v),v\right),\\
    h_\alpha(t,x,v):=&f_\alpha\left(t,x+t\widehat{v_\alpha}+\frac{e_\alpha}{v^0_\alpha}\log(t)(\mathcal{C}_{v_\alpha}-\widehat{v_\alpha}\mathcal{D}_{v_\alpha}),v\right)=g_\alpha\left(t,x+\frac{e_\alpha}{v^0_\alpha}\log(t)(\mathcal{C}_{v_\alpha}-\widehat{v_\alpha}\mathcal{D}_{v_\alpha}),v\right).
\end{align}
\begin{proposition}
    For any $\alpha$, all $t\geq 0$ and all $(x,v)\in\R^3_x\times\R^3_v$, we have 
    \begin{equation}
    \left|f_\alpha\left(t,x+t\widehat{v_\alpha}+\frac{e_\alpha}{v^0_\alpha}\log(t)(\mathcal{C}_{v_\alpha}-\widehat{v_\alpha}\mathcal{D}_{v_\alpha}),v\right)-f_{\alpha\infty}(x,v)\right|\lesssim \frac{\log^6(2+t)}{(2+t)}.
    \end{equation}
\end{proposition}
\begin{proof}
    We directly have, thanks to the above equation \eqref{equation_finale_derivee_en_temps_g_alpha} on $\partial_t g^\alpha$,
    \begin{align*}
        \partial_t h^\alpha(t,x,v)=&\frac{1}{t}\frac{e_\alpha}{m_\alpha v^0}(\mathcal{C}_v-\widehat{v}\mathcal{D}_v)\cdot(\nabla_x g^\alpha)\left(t,x+\frac{e_\alpha}{m_\alpha v^0}\log(t)(\mathcal{C}_v-\widehat{v}\mathcal{D}_v),v\right)+(\partial_t g^\alpha)\left(t,x+\frac{e_\alpha}{m_\alpha v^0}\log(t)(\mathcal{C}_v-\widehat{v}\mathcal{D}_v),v\right)\\
        =& \frac{1}{t}\frac{e_\alpha}{m_\alpha v^0} \Big[(\mathcal{C}_v-\widehat{v}\mathcal{D}_v)-\left(\widehat{v}(\mathbb{L}(v)\cdot\widehat{v})-\mathbb{L}(v)\right)\Big]\cdot\nabla_x g^\alpha+O\left(\frac{\log^6(2+t)}{(2+t)^2}\right)\\
        =&O\left(\frac{\log^6(2+t)}{(2+t)^2}\right).
    \end{align*}
    Consequently, there exists ${f^\alpha}_\infty\in C^0(\R^3_x\times\R^3_v)$ such that
    \begin{equation}
        \left|f^\alpha\left(t,x+t\widehat{v}+\frac{e_\alpha}{m_\alpha v^0}\log(t)(\mathcal{C}_v-\widehat{v}\mathcal{D}_v),v\right)-{f^\alpha}_\infty(x,v)\right|\lesssim \frac{\log^6(2+t)}{(2+t)}.
    \end{equation}
    This directly implies the result, with $f_{\alpha\infty}(x,v):={f^\alpha}_\infty(x,v_\alpha)$, where we recall $v^0_\alpha=\sqrt{m_\alpha^2+|v|^2}$ and $v_\alpha=\frac{v}{m_\alpha}$.
\end{proof}
It remains to prove the estimate of Theorem \ref{theoreme_principal_scattering_modifie}. For this, write $\widetilde{\mathcal{C}}_v=\mathcal{C}_{v}-\widehat{v}\mathcal{D}_{v}$ and let
\begin{align*}
    h_\alpha(t,x,v)&=f_\alpha\left(t,x+t\widehat{v_\alpha}+\frac{e_\alpha}{v^0_\alpha}\log(t)\widetilde{\mathcal{C}}_{v_\alpha},v\right),\\
    \widetilde{h}_\alpha(t,x,v)&:=f_\alpha\left(tv^0_\alpha+\frac{e_\alpha}{v^0_\alpha}\log(t)\mathcal{D}_{v_\alpha},x+tv+\frac{e_\alpha}{v^0_\alpha}\log(t)\mathcal{C}_{v_\alpha},v\right).
\end{align*}

\begin{remark}
    Notice that for $\widetilde{h}_\alpha$ to be well-defined, one needs to have $tv^0_\alpha+\dfrac{e_\alpha}{v^0_\alpha}\log(t)\mathcal{D}_{v_\alpha}\geq 0$. Since $|v|\leq \beta$, this holds whenever $t$ is large enough.
\end{remark}

 We already know that 
\begin{equation*}
    |h_\alpha(tv^0_\alpha,x,v)-f_{\alpha\infty}(x,v)|\lesssim \frac{\log^6(2+t)}{2+t}.
\end{equation*}
Moreover,
\begin{align*}
    \widetilde{h}_\alpha(t,x,v)&=f_\alpha\left(tv^0_\alpha+\frac{e_\alpha}{v^0_\alpha}\log(t)\mathcal{D}_{v_\alpha},x+(tv^0_\alpha+\frac{e_\alpha}{v^0_\alpha}\log(t)\mathcal{D}_{v_\alpha})\widehat{v_\alpha}+\frac{e_\alpha}{v^0_\alpha}\log(t)\widetilde{\mathcal{C}}_{v_\alpha},v\right)\\
    &=g_\alpha\left(tv^0_\alpha+\frac{e_\alpha}{v^0_\alpha}\log(t)\mathcal{D}_{v_\alpha},x+\frac{e_\alpha}{v^0_\alpha}\log(t)\widetilde{\mathcal{C}}_{v_\alpha},v\right).
\end{align*}
Now we know that, thanks to \eqref{equation_estimee_EB_Glassey} and Proposition \ref{proposition_estimee_gradv_g}, $|\partial_t g_\alpha(t,\cdot)|\lesssim \frac{1}{2+t}$. Hence, by the mean value theorem, we derive, for $t$ large enough, 
\begin{align*}
    \widetilde{h}_\alpha(t,x,v)&=g_\alpha\left(tv^0_\alpha,x+\frac{e_\alpha}{v^0_\alpha}\log(t)\widetilde{\mathcal{C}}_{v_\alpha},v\right)+O\left(\frac{\log(2+t)}{2+t}\right).
\end{align*}
Finally, using the expression of $h_\alpha$, we obtain
\begin{align*}
    \widetilde{h}_\alpha(t,x,v)&=g_\alpha\left(tv^0_\alpha,x-\frac{e_\alpha}{v^0_\alpha}\log(v^0_\alpha)\widetilde{\mathcal{C}}_{v_\alpha}+\frac{e_\alpha}{v^0_\alpha}\log(tv^0_\alpha)\widetilde{\mathcal{C}}_{v_\alpha},v\right)+O\left(\frac{\log(2+t)}{2+t}\right)\\
    &=h_\alpha\left(tv^0_\alpha,x-\frac{e_\alpha}{v^0_\alpha}\log(v^0_\alpha)\widetilde{\mathcal{C}}_{v_\alpha},v\right)+O\left(\frac{\log(2+t)}{2+t}\right)\\
    &=f_{\alpha\infty}\left(x-\frac{e_\alpha}{v^0_\alpha}\log(v^0_\alpha)\widetilde{\mathcal{C}}_{v_\alpha},v\right)+O\left(\frac{\log^6(2+t)}{2+t}\right).
\end{align*}
So we derive the estimate of Theorem \ref{theoreme_principal_scattering_modifie} by taking $\widetilde{f}_{\alpha\infty}(x,v)=f_{\alpha\infty}\left(x-\dfrac{e_\alpha}{v^0_\alpha}\log(v^0_\alpha)\widetilde{\mathcal{C}}_{v_\alpha},v\right)$. However, to prove Theorem \ref{theoreme_principal_scattering_modifie}, it remains to show that $\widetilde{f}_{\alpha\infty}$ has a compact support.

\begin{proposition}
    \label{proposition_h_support_compact}
    There exists $\mathcal{T}> 0$ and a constant $C>0$ such that, for all $t\geq \mathcal{T}$ 
    \begin{equation}
        \supp h^\alpha(t,\cdot) \subset \{ (x,v)\in\R^3_x\times\R^3_v,\, |x|\leq C,\, |v|\leq C\},
    \end{equation}
    i.e. $h^\alpha(t,\cdot)$ is compactly supported, uniformly in $t$.
\end{proposition}

\begin{proof}
    Recall the expression of $h^\alpha$ 
    \begin{equation*}
        h^\alpha(t,x,v)=f^\alpha\left(t,x+t\widehat{v}+\frac{e_\alpha}{m_\alpha v^0}\log(t)(\mathcal{C}_v-\widehat{v}\mathcal{D}_v),v\right)=g^\alpha\left(t,x+\frac{e_\alpha}{m_\alpha v^0}\log(t)(\mathcal{C}_v-\widehat{v}\mathcal{D}_v),v\right),
    \end{equation*}
    and suppose $g^\alpha\left(t,x+\frac{e_\alpha}{m_\alpha v^0}\log(t)(\mathcal{C}_v-\widehat{v}\mathcal{D}_v),v\right)\neq 0$. Consider the characteristics 
    \begin{equation*}
        \mathcal{X}(s)=\mathcal{X}\left(s,t,x+\frac{e_\alpha}{m_\alpha v^0}\log(t)(\mathcal{C}_v-\widehat{v}\mathcal{D}_v),v\right),\quad \mathcal{V}(s)=\mathcal{V}\left(s,t,x+\frac{e_\alpha}{m_\alpha v^0}\log(t)(\mathcal{C}_v-\widehat{v}\mathcal{D}_v),v\right).
    \end{equation*}
    Now recall the ODE \eqref{ODE_characteristics_g} satisfied by $(\mathcal{X},\mathcal{V})$.  We have $|\dot{\mathcal{V}}(s)|\lesssim |L(s,\mathcal{X}(s)+s\widehat{\mathcal{V}}(s),\mathcal{V}(s))|\lesssim (s+2)^{-2}$, according to Lemma \ref{lemme_support_f^alpha_en_x} and \eqref{equation_estimee_EB_Glassey}. Thus, there exists $v_\infty\in\R^3_v$ such that 
    \begin{equation*}
        |\mathcal{V}(s)-v_\infty|\lesssim \frac{1}{2+s}.
    \end{equation*}
    Now, using \eqref{equation_estimee_EB_Glassey}--\eqref{equation_estimee_gradEB_Glassey}, the mean value theorem and the above equation, we derive
    \begin{align*}
        \dot{\mathcal{X}}(s)&=\dfrac{s}{\mathcal{V}^0(s)}\dfrac{e_\alpha}{m_\alpha}\left[\widehat{\mathcal{V}}(s)\left(L(s,\mathcal{X}(s)+s\widehat{\mathcal{V}}(s),\mathcal{V}(s))\cdot\widehat{\mathcal{V}}(s)\right)-L(s,\mathcal{X}(s)+s\widehat{\mathcal{V}}(s),\mathcal{V}(s))\right]\\
        &=\dfrac{s}{v_\infty^0}\dfrac{e_\alpha}{m_\alpha}\left[\widehat{v_\infty}\left(L(s,\mathcal{X}(s)+s\widehat{v_\infty},v_\infty)\cdot\widehat{v_\infty}\right)-L(s,\mathcal{X}(s)+s\widehat{v_\infty},v_\infty)\right]+O\left(\frac{\log(2+s)}{(2+s)^2}\right).
    \end{align*}
    This grants us, by applying Proposition \ref{proposition_estimee_finale_E_E_bb},
    \begin{align*}
        \dot{\mathcal{X}}(s) &=\dfrac{1}{sv_\infty^0}\dfrac{e_\alpha}{m_\alpha}\left[\widehat{v_\infty}\left(\mathbb{L}(v_\infty)\cdot\widehat{v_\infty}\right)-\mathbb{L}(v_\infty)\right]+O\left(\frac{\log^6(2+s)}{(2+s)^2}\right)\\
        &=\frac{1}{s}\frac{e_\alpha}{m_\alpha v^0_\infty}\widetilde{\mathcal{C}}_{v_\infty}+O\left(\frac{\log^6(2+s)}{(2+s)^2}\right).
    \end{align*}
    By integrating we derive, for $t\geq 1$, 
    \begin{equation*}
        \left|\mathcal{X}(t)-\frac{e_\alpha}{m_\alpha v^0_\infty}\log(t)\widetilde{\mathcal{C}}_{v_\infty}\right|\leq \mathcal{X}(1)+C.
    \end{equation*}
    Finally, one can notice that since $|\mathcal{V}(s)-v_\infty|\lesssim (2+s)^{-1}$ and $|\mathcal{V}(s)|\leq \beta_\alpha$, we have $|v_\infty|\leq \beta_\alpha$. Then, it yields
    \begin{align*}
        |\E(v_\infty)-\E(\mathcal{V}(t))|\leq &|\E(v_\infty)-t^2E(t,t\widehat{v_\infty})| + |\E(\mathcal{V}(t))-t^2E(t,t\widehat{\mathcal{V}}(t))| + t^2|E(t,t\widehat{v_\infty})-E(t,t\widehat{\mathcal{V}}(t))|\\
        \lesssim & \frac{\log^6(2+t)}{2+t}.
    \end{align*}
    By deriving the same estimate for $\mathbb{B}$, we obtain for $t$ large enough
    \begin{equation*}
        \log(t)|\widetilde{\mathcal{C}}_{v_\infty}-\widetilde{\mathcal{C}}_{\mathcal{V}(t)}|\lesssim 1.
    \end{equation*}
    So finally,
    \begin{equation*}
        \left|\mathcal{X}(t)-\frac{e_\alpha}{m_\alpha\mathcal{V}^0(t)}\log(t)\widetilde{\mathcal{C}}_{\mathcal{V}(t)}\right|\lesssim 1,
    \end{equation*}
    i.e. since $\mathcal{V}(t)=v$ and $\mathcal{X}(t)=x+\log(t)\frac{e_\alpha}{m_\alpha v^0}\widetilde{\mathcal{C}}_v$ we have
    \begin{equation*}
        |x|\lesssim 1.
    \end{equation*}

\end{proof}

\begin{corollary}
    The functions $h_\alpha$ and $\widetilde{h}_\alpha$ are compactly supported. That is $\exists M>0$, such that $\forall t>0$, 
    \begin{equation*}
        \supp h_\alpha(t,\cdot) \subset \{ (x,v)\in\R^3_x\times\R^3_v,\, |x|\leq M,\, |v|\leq M\},
    \end{equation*}
    \begin{equation*}
        \supp \widetilde{h}_\alpha(t,\cdot) \subset \{ (x,v)\in\R^3_x\times\R^3_v,\, |x|\leq M,\, |v|\leq M\}.
    \end{equation*}
\end{corollary}
\begin{proof}
    Since $h_\alpha(t,x,v)=h^\alpha(t,x,v_\alpha)$, the first result is immediate. Now, notice that $\widetilde{h}_\alpha(t,x,v)=h_\alpha(s,y,v)$ with 
    \begin{equation*}
        s=tv^0_\alpha+\log(t)\frac{e_\alpha}{v^0_\alpha}\mathcal{D}_{v_\alpha},\quad y=x+\frac{e_\alpha}{v^0_\alpha}(\log(t)-\log(s))\widetilde{\mathcal{C}}_{v_\alpha}.
    \end{equation*}
    Assume $\widetilde{h}_\alpha(t,x,v)\neq 0$, then, by Proposition \ref{proposition_h_support_compact}, $|y|\leq C$ and $|x|\lesssim C +|\log(t)-\log(s)|$. Moreover, \\$|\log(t)-\log(s)|\xrightarrow[t\rightarrow +\infty]{}\log(v^0_\alpha)$ so $t\mapsto|\log(t)-\log(s)|$ is bounded, uniformly in $v$ on the support of $f^\alpha$, by $\kappa$. Then $|x|\lesssim C+\kappa$.
\end{proof}

\begin{corollary}
    The functions ${f^\alpha}_\infty, f_{\alpha\infty}$ and $\widetilde{f}_{\alpha\infty}$ are all compactly supported.
\end{corollary}



\printbibliography

\end{document}